\documentclass[12pt,leqno]{amsart}
\usepackage{amscd}
\usepackage{amssymb}
\usepackage{amsfonts}
\usepackage{latexsym}
\usepackage{verbatim}

\theoremstyle{plain}
\newtheorem{theorem}{Theorem}[section]

\newtheorem{lemma}[theorem]{Lemma}
\newtheorem{prop}[theorem]{Proposition}

\newtheorem{ex}[theorem]{Example}

\newcommand\R{{\mathbb R}}

\title[Calabi-Yau equation on the Kodaira-Thurston manifold]{On the Calabi-Yau equation in the Kodaira-Thurston manifold}

\author{Luigi Vezzoni}

\date{\today}
\address{Dipartimento di Matematica G. Peano, Universit\`a di Torino, Via Carlo Alberto 10, 10123 Torino, Italy.} \email{luigi.vezzoni@unito.it}

\thanks{This work was supported by the project FIRB ``Geometria Differenziale e teoria geometrica delle funzioni'' and by G.N.S.A.G.A. of I.N.d.A.M}

\begin{document}

\maketitle

\begin{abstract}
We review some previous results about the Calabi-Yau equation on the Kodaira-Thurston manifold equipped with an invariant almost-K\"ahler structure and assuming the volume form $T^2$-invariant. In particular, we observe that under some restrictions the problem is reduced to a Monge-Amp\`ere equation  by using the ansatz $\tilde \omega=\Omega-dJdu+da$, where $u$ is a $T^2$-invariant function and $a$ is a $1$-form depending on $u$. Furthermore, we extend our analysis to non-invariant almost-complex structures by considering some basic cases and we finally take into account a generalization to higher dimensions.  
\end{abstract}

\section{Introduction}
The {\em Calabi-Yau problem} in $4$-dimensional almost-K\"ahler manifolds is a PDEs system arising from the generalization of the classical Calabi-Yau theorem to the non-K\"ahler setting.  

The Calabi-Yau theorem \cite{Yau} states that on a compact K\"ahler manifold $(X,J,\Omega)$ for every  smooth function $F\colon X\to \mathbb R$ such that 
\begin{equation}\label{intF}
\int_X {\rm e}^F\,\Omega^n=\int_X \Omega^n
\end{equation}
there always exists a unique K\"ahler form $\tilde \omega$ on $(X,J)$ satisfying  
\begin{equation}\label{CYS}
[\tilde \omega]=[\Omega]\,,\quad \tilde \omega^n={\rm e}^F\,\Omega^n\,. 
\end{equation}
An analogue problem still makes sense in the almost-K\"ahler case, when $J$ is merely an
almost-complex structure and $\Omega$ is a $J$-compatible symplectic form. It turns out that in this more general context, the PDEs system arising from \eqref{CYS} is overdetermined for $n\geq 3$, while it is elliptic in dimension $4$ (see \cite{D}). Consequently, the Calabi-Yau problem is mainly studied in $4$-dimensional almost-K\"ahler manifolds (see \cite{REN,JDG,TWWY,TWY,TW,W} and the references therein). 

The study of  the problem is strongly motivated by a project of Donaldson involving compact symplectic $4$-manifolds (see \cite{D}).  The project is based on a conjecture stated in \cite{D} and partially confirmed by Taubes in \cite{taubes}.

In \cite{W} Weinkove attacked the problem by introducing a {\em symplectic potential}. Indeed, given two almost-K\"ahler forms 
$\Omega$ and $\tilde \omega$ on a compact almost-complex manifold $(X,J)$ satisfying $[\Omega]=[\tilde \omega]$ there always exists a function $u$, called the {\em symplectic potential}, such that 
$$
(\tilde \omega-\Omega)\wedge \tilde \omega=-dJdu\wedge \tilde \omega\,.  
$$
In terms of $u$ one can always write 
$$
\tilde \omega= \Omega-dJdu+da\,, 
$$
where $a$ is a $1$-form which can be assumed co-closed with respect to the co-differential induced by $\tilde \omega$ (in this way $a$ is unique up addiction of $harmonic$ $1$-forms). 

Weinkove proved that in order to show the solvability of the Calabi-Yau problem \eqref{CYS} it's enough to provide an a priori estimate on the $C^0$-norm of the almost-K\"ahler potential (see theorem 1 in  \cite{W}); that can be always done if the $L^1$-norm of the Nijenhuis tensor of $J$ is small enough (see theorem 2 in \cite{W}). 
   
In \cite{TW} Tosatti and Weinkove studied the Calabi-Yau problem on the Kodaira-Thurston manifold $(M,\Omega_0,J_0)$ showing that under the assumption on the initial datum $F$ to be invariant 
by the action of a $2$-dimensional torus the problem has a unique solution.  
The Kodaira-Thurston manifold $M$ is a $4$-dimensional $2$-step nilmanifold 
carrying a natural almost-K\"ahler structure and it can be viewed as a torus bundle over a torus (more precisely $M$ is an $S^1$-bundle over a $3$-dimensional torus). 

In \cite{FLSV} it is observed that if $F$ is $T^2$-invariant, then  \eqref{CYS} on the Kodaira-Thrurston manifold $M$ can be rewritten in terms of the 
Monge-Amp\`ere equation 
$$
(1+u_{xx})(1+u_{yy})-u_{xy}^2={\rm e}^F
$$
on the $2$-dimensional torus $\mathbb T^{2}_{xy}$ and the Tosatti-Weinkove result in \cite{TW} can be alternatively obtained by applying a result of Y.Y. Li in \cite{YY}. A similar approach was then adopted in \cite{REN,FLSV} in order to study the Calabi-Yau problem in every $4$-dimensional torus bundle over a torus equipped with an invariant almost-K\"ahler structure. In this more general case the equation writes in terms of a \lq\lq modified\rq\rq\, Monge-Amp\`ere equation which is still solvable. 
Furthermore, in \cite{JDG} it is studied the equation on the Kodaira-Thurston manifold when $F$  is $S^1$-invariant (instead of $T^2$-invariant as in the previous papers). It turns out that in this last case the Calabi-Yau problem writes as a PDE on the $3$-dimensional torus $\mathbb T^3_{xyt}$ which is not of Monge-Amp\`ere type anymore. 

In this paper we review some results in \cite{FLSV} showing that when the projection is Lagrangian, the reduction of the Calabi-Yau problem on the Kodaira-Thurston manifold to a scalar PDE can be obtained by setting  
$$
\tilde \omega=\Omega+d(-Ju+u\gamma_1+u_{y}\gamma_2)
$$
where $\gamma_1$ and $\gamma_2$ are suitable invariant forms depending on $(\Omega,J)$, $u$ is in the same space of $F$ and $y$ is a coordinate on the base.  

In section \ref{S3} we study the Calabi-Yau equation on $(M,\Omega_0)$ for $S^1$-invariant almost complex structures $J$ compatible to $\Omega_0$. Under some strong restrictions on $J$, the equation can be still reduced to a PDE in a single unknown function. In section \ref{S4} we prove the solvability of the arising equations in some special cases leaving the more general cases for an eventually future work. 

In the last section we consider a generalization of the previous sections to $2$-step nilmanifold in higher dimensions.  

\medskip
\noindent {\bf A remark on the notation.} {If $P$ is an  $m$-torus bundle over an $n$-torus, we denote by $\mathbb T^n$ the base of $P$ and by $T^m$ the principal fiber, in order to distinguish the base and the fibers.}

\bigbreak\noindent{\it Acknowledgments.} The research of the present  paper was originated by some discussions during  \lq\lq {\em The $4^{\rm th}$ workshop on complex Geometry and Lie groups}\rq\rq\,  hold in Nara from the 22nd to the 26th of March 2016. The author thanks Anna Fino, Ryushi Goto and Keizo Hasegawa for the
kind invitation.  Moreover, the author is very grateful to Ernesto Buzano, Giulio Ciraolo, Valentino Tosatti and Michela Zedda for useful conversations.

\section{Calabi-Yau equations on the Kodaira-Thurston manifold}

In this section we review some results in \cite{REN,JDG,FLSV} about the Calabi-Yau equation on the Kodaira-Thurston manifold.  
The {\em Kodaira-Thurston manifold} is a compact $2$-step nilmanifold $M$ defined as the quotient $M=\Gamma \backslash G$, where $G$ is the Lie group given by $\R^4$ in the variables $(x_1,x_2,y_1,y_2)$ with the multiplication 
$$
(x_1,x_2,y_1,y_2)\cdot(x_1',x_2',y_1',y_2')=(x_1+x_1',x_2+x_2',y_1+y_1',y_2+y_2'+x_1x_2')
$$
and $\Gamma$ is the co-compact lattice given by $\mathbb Z^4$ with the induced multiplication. Alternatively $M$ can be defined as the product $M=\Gamma_0 \backslash {\rm Nil}^3\times S^1$, where ${\rm Nil}^3$  is  the $3$-dimensional  real Heisenberg group
\begin{equation*}
{\rm Nil}^3=\left\{\left[\begin{smallmatrix}1&x&z\\0&1&y\\0&0&1\end{smallmatrix}\right]\mid x,y,z \in \mathbb R\right\}
\end{equation*}
and $\Gamma_0$ is  the lattice in ${\rm Nil}^3$ of matrices having integers entries. 
 $M$ has a natural structure of principal $S^1$-bundle over a $3$-dimensional torus $\mathbb T^3$ induced by the map $[x_1,x_2,y_1,y_2]\mapsto [x_1,x_2,y_1]$ and it is parallelizable. 
A global co-frame on $M$ is for instance given by 
$$
e^1=dx_1,\quad e^2=dx_2\,,\quad f^1=dy_1\,,\quad f^2=dy_2-x_1dx_2\,.
$$ 
For such co-frame we have 
$$
de^1=de^2=df^1=0\,,\quad df^2=-e^1\wedge e^2
$$
and its dual basis is given  by 
$\{\partial_{x_1}\,, \partial_{x_2}+x_1\partial_{y_2}\,, \partial_{y_1},-\partial_{y_2}\}.$
Furthermore, $M$ has the \lq\lq natural\rq\rq almost-K\"ahler structure $(\Omega_0,J_0)$ given by the symplectic form  
\begin{equation}
\label{standard}
\Omega_0=e^1\wedge f^1+e^2\wedge f^2
\end{equation}
and the Riemannian metric 
\begin{equation}\label{SM}
g_0=e^1\otimes e^1+f^1\otimes f^1+e^2\otimes e^2+f^2\otimes f^2\,.
\end{equation}
The following proposition is proved in \cite{JDG}
\begin{prop}\label{propBFV}
Let $u\colon M\to\R$ be an $S^1$-invariant function and 
$$
\alpha:=-J_0du-ue^1. 
$$
Then 
$$
d\alpha \mbox{ is of type }(1,1)
$$ 
and 
\begin{equation}\label{CYspecific}
(\Omega_0+d\alpha)^2=\left(\det(I+\mathcal{A}(u))-u_{x_2y_1}^2\right)\,\Omega_0^2\,,
\end{equation}
where $I$ is the identity $2\times 2$ matrix and 
\begin{equation}\label{A(u)}
\mathcal{A}(u)=\left(
\begin{array}{cc}
u_{x_1x_1}+u_{y_1y_1}+u_{y_1} & u_{x_1x_2}\\
u_{x_1x_2} & u_{x_2x_2}
\end{array}
\right)\,. 
\end{equation}
\end{prop}
\begin{proof}
Let $u\colon M\to \R$ be an $S^1$-invariant function. Then  
$$
du=u_{x_1}e^1+u_{x_2}e^2+u_{y_1}f^1
$$
and 
$$
-J_0du=u_{x_1}f^1+u_{x_2}f^2-u_{y_1}e^1
$$
and 
$$
-dJ_0du=\sum_{i,j=1}^2 u_{x_ix_j} e^i\wedge f^j+u_{x_2y_1}e^1\wedge e^2+u_{x_2y_1}f^1\wedge f^2
+u_{y_1y_1}e^1\wedge f^1-u_{x_2} e^1\wedge e^2\,. 
$$
Therefore, if $\alpha=-J_0du-ue^1$, we have 
\begin{multline*}
d\alpha=-dJ_0du-du\wedge e^1\\=\sum_{i,j=1}^2 u_{x_ix_j} e^i\wedge f^j+u_{x_2y_1}e^1\wedge e^2+u_{x_2y_1}f^1\wedge f^2
+u_{y_1y_1}e^1\wedge f^1+u_{y_1}e^1\wedge f^1
\end{multline*}
which is a form of type $(1,1)$ with respect to $J_0$. Formula \eqref{CYspecific} follows from a straightforward computation. 
\end{proof}

Proposition \ref{propBFV} is useful in the study of the Calabi-Yau problem on $(M,\Omega_0,J_0)$.  Indeed,  let $F\colon M\to \mathbb R$ be an $S^1$-invariant function satisfying $\int_M {\rm e}^F\, \Omega_0^2=1$
and consider the 
Calabi-Yau equation $(\Omega_0+d\alpha)^2={\rm e}^F\,\Omega_0^2$ on $(M,\Omega_0,J_0)$. In view of proposition \ref{propBFV}, we can study the Calabi-Yau problem by introducing the ansatz 
$$
\alpha=-J_0du-u e^1
$$  
where $u$ is an  unknown $S^1$-invariant map. In this way the Calabi-Yau problem reduces to the single equation 
\begin{equation}\label{det A}
\det(I+\mathcal{A}(u))-u_{x_2y_1}^2={\rm e}^F\,,
\end{equation}
on the $3$-dimensional torus  $\mathbb T^3_{x_1x_2y_1}$, where $\mathcal{A}(u)$ is given by \eqref{A(u)}.  The main result in \cite{JDG} is the following 
\begin{theorem}
Equation \eqref{det A} has a solution for every $S^1$-invariant initial  datum $F\colon M\to \R$ satisfying $\int_M {\rm e}^F\,\Omega_0^2=1$. Consequently the Calabi-Yau problem $(\Omega_0+d\alpha)^2={\rm e}^F\,\Omega_0^2 $ has a unique solution for every $S^1$-invariant function $F\colon M\to \mathbb R$.  
\end{theorem}

Special cases of equation \eqref{det A} occur when we see $M$ as a $2$-torus bundle over a $2$-dimensional torus and we assume $F$ depending only on the coordinates of the base. Those cases correspond to assume $F$ depending either on $(x_1,x_2)$ or on $(x_2,y_1)$ (the case $F=F(x_1,y_1)$ is equivalent to $F=F(x_2,y_1)$). 

\medskip
 If $F=F(x_1,x_2)$, we can assume $u$ depending only on $(x_1,x_2)$ and \eqref{det A} reduces to the\rq\, Monge-Amp\`ere type equation  
\begin{equation}\label{STDMA}
(1+u_{x_1x_1})(1+u_{x_2x_2})-u_{x_1x_2}^2={\rm e}^F
\end{equation}
on the $2$-dimensional torus $\mathbb T^2_{x_1x_2}$. This equation has a solution in view of a theorem of Y.Y. Li (see \cite{YY}).  Note that in this case the solution $u$ to \eqref{STDMA} is an almost-K\"ahler potential of $\tilde\omega= \Omega_0+d\alpha$ with respect to $\Omega_0$. Indeed,  
$$
\begin{aligned}
\tilde \omega=(1+u_{x_1x_1} )e^1\wedge f^1+(1+u_{x_2x_2}) e^2\wedge f^2+u_{x_1x_2} e^1\wedge f^2+u_{x_1x_2} f^1\wedge e^2
\end{aligned}
$$
and 
$$
\tilde \omega-\Omega_0=-dJ_0du+da
$$
where 
$$
a=-ue^1\,. 
$$
Hence $d a=u_{x_2}e^1\wedge e^2$ and 
$$
\tilde \omega\wedge da=0
$$
which implies 
$$
(\tilde \omega-\Omega_0)\wedge \tilde \omega =-dJ_0du\wedge\tilde \omega\,. 
$$

\vspace{0.2cm}
 If $F=F(x_2,y_1)$, we assume $u$ depending only on $(x_2,y_1)$ and \eqref{det A} reduces to the \lq\lq modified\rq\rq\, Monge-Amp\`ere equation
\begin{equation}\label{GenMA}
(1+u_{y_1y_1}+u_{y_1})(1+u_{x_2x_2})-u_{x_2y_1}^2={\rm e}^F
\end{equation}
on the $2$-dimensional torus $\mathbb T^2_{x_2y_1}$. The existence of a solution to this last equation was proved in \cite{FLSV}.  
Note that in this case 
$$
\tilde \omega=(1+u_{y_1y_1}+u_{y_1})e^1\wedge f^1+ (1+u_{x_2x_2})e^2\wedge f^2+u_{x_2y_1}e^1\wedge e^2+u_{x_2y_1}f^1\wedge f^2
$$
and if $u$ solves \eqref{GenMA}, then 
$$
d\alpha=-dJ_0du+da\,,
$$
where $d a=-u_{x_2}e^1\wedge e^2-u_{y_1} e^1\wedge f^1$. Therefore 
$$
da\wedge \tilde \omega \neq 0
$$
and $u$ is not  an almost-K\"ahler potential. 

\medskip 
Next, we take into account the Calabi-Yau problem on $M$ viewed as a $2$-torus bundle over a $2$-torus equipped with an   invariant Lagrangian almost-K\"ahler structure $(\Omega,J)$ and we assume $F$ defined on the base. Here by {\em Lagrangian} we mean that the fibers of the fibration are Lagrangian submanifolds.
\begin{prop}\label{PROP1}
Let $(\Omega,J)$ be an {\em invariant} almost-K\"ahler structure on $M$.  Then there exist real numbers $\mu_1$ and $\mu_2$ and an invariant $1$-form $\beta$ such that if 
$u=u(x_1,x_2)$ is a smooth function on $M$,  then 
$$
\alpha=-Jdu+\mu_1 u\, e^1 -\mu_2 u\,e^2-u_y\beta
$$
is such that $d\alpha$ is of type $(1,1)$. Moreover 
\begin{equation}\label{KTx1x1}
(\Omega+d\alpha)^2=\frac{1}{l_1l_2}\left((l_{1}+u_{x_1x_1})(l_{2}+u_{x_2x_2})-(u_{x_1x_2})^2\right)\,\Omega^2\,.
\end{equation}
where $l_1$ and $l_2$ are positive real numbers.   
\end{prop}
\begin{proof}
We set $x_1=x$ and $x_2=y$ in order to simplify the notation. 
We can find an invariant Hermitian coframe  $\{\alpha^1,\alpha^2,\beta^1,\beta^2\}$ on $M$ such that 
$$
\Omega=\alpha^1\wedge \beta^1+\alpha^2\wedge \beta^2
$$
and 
$$
dx=A\alpha^1\,,\quad dy=B\alpha^1+C\alpha^2\,. 
$$
Note that $dx\wedge dy=AC \alpha^1\wedge\alpha^2$ and we can  write 
$$
d\beta^1=\lambda_1dx\wedge dy\,,\quad d\beta^2=\lambda_2dx\wedge dy
$$
for some $\lambda_1,\lambda_2$ in $\R$. 
Now 
$$
du=u_xdx+u_ydy=(Au_x+Bu_y)\alpha^1+Cu_y\alpha^2
$$
and 
$$
-Jdu=(Au_x+Bu_y)\beta^1+Cu_y\beta^2
$$
So 
\begin{multline*}
-dJdu=Au_{xx}dx\wedge\beta^1+Au_{xy}dy\wedge\beta^1+Cu_{xy}dx\wedge \beta^2+Cu_{yy}dy\wedge \beta^2+d(\gamma+Bu_{y}\beta^1)
\end{multline*}
where 
$$
\gamma=\lambda_1Au\, dy-\lambda_2Cu\, dx\,. 
$$
Hence 
\begin{multline*}
-dJdu=A^2u_{xx}\alpha^1\wedge \beta^1+ABu_{xy}\alpha^1\wedge \beta^1+ACu_{xy}\alpha^2\wedge\beta^1+ACu_{xy}\alpha^1\wedge \beta^2\\+BCu_{yy}\alpha^1\wedge \beta^2+C^2u_{yy}\alpha^2\wedge \beta^2+d(Bu_y\alpha^2+\gamma)\,. 
\end{multline*}
which implies that 
$$
\alpha=-Jdu-Bu_y\alpha^2-\gamma
$$ 
is such that $d\alpha$ is of type $(1,1)$. 

Moreover,  
\begin{multline*}
(\Omega+d\alpha)^2=\left((1+A^2u_{xx})(1+C^2u_{yy})-(ACu_{xy})^2\right)\,\Omega^2\\
=\frac{1}{l_1l_2}\left((l_{1}+u_{xx})(l_{2}+u_{yy})-(u_{xy})^2\right)\,\Omega^2
\end{multline*}
where $l_1=1/A^2$ and $l_2=1/C^2$ and the claim follows. 
\end{proof}
%
%
\begin{prop}\label{PROP2}
Let $(\Omega,J)$ be an {\em invariant} almost-K\"ahler structure on $M$ 
which is {\em Lagrangian} with respect to the fibration $[x_1,x_2,y_1,y_2]\mapsto [x_2,y_1]$. There exist invariant $1$-forms $\gamma^1, \gamma^2$ such that  if $u=u(x_2,y_1)$ is a smooth function on $M$, then 
$$
\alpha=-Jdu+u\gamma^1+u_{y_1}\gamma^2
$$
is such that $d\alpha$ is of type $(1,1)$. Moreover 
\begin{equation}\label{KTx2y1}
(\Omega+d\alpha)^2=\frac{1}{l_1l_2}\left((l_1+u_{x_2x_2})(l_2+u_{y_1y_1}+m_1u_{x_2}+m_2u_{y_1})-(u_{x_2y_1})^2\right)\, \Omega^2
\end{equation}
where $l_1,l_2,m_1,m_2\in \R$ and $l_1,l_2<0$. 
\end{prop}
\begin{proof}
First of all we use that $(\Omega,J)$ is an invariant almost-K\"ahler structure on $M$ which is Lagrangian with respect to  $[x_1,x_2,y_1,y_2]\mapsto [x_2,y_1]$, then there exists an invariant Hermitian co-frame $\{\alpha^1,\alpha^2,\beta^1,\beta^2\}$ on $M$ such that 
$$
\Omega=\alpha^1\wedge \beta^1+\alpha^2\wedge \beta^2
$$
and 
$$
\alpha^2\in\langle e^2\rangle\,,\quad \beta^1\in \langle e^2,f^1\rangle \,,\quad \alpha^1\in \langle e^1,e^2,f^1\rangle 
$$
(see lemma 5.1 in \cite{FLSV}). 
In this way 
$$
dx_2=A\alpha^2\,,\quad dy_1=B\alpha^2+C\beta^1\,,\quad d\beta^2=\lambda\,\alpha^1\wedge \beta^1+\mu\alpha^2\wedge\beta^1
$$
for some $A,B,C,\lambda,\mu\in \mathbb R$. In order to semplify the notation we set $x_2=x$ and $y_1=y$. Then
$$
du=u_x dx+u_y dy=Au_x\alpha^2+u_y(B \alpha^2+C\beta^1)=(Au_x+Bu_y)\alpha^2+Cu_y\beta^1
$$
and 
$$
-Jdu=-(Au_x+Bu_y)\beta^2+Cu_y\alpha^1\,.
$$
So 
\begin{multline*}
-dJdu=-Au_{xx}dx\wedge\beta^2-Au_{xy}dy\wedge \beta^2+Cu_{xy}dx\wedge \alpha^1\\
+Cu_{yy}dy\wedge \alpha^1-(Au_x+Bu_y)\left(\lambda\,\alpha^1\wedge \beta^1+\mu\alpha^2\wedge\beta^1\right)-d(Bu_{y}\beta^2)
\end{multline*}
i.e., 
\begin{multline*}
-dJdu=-A^2u_{xx}\alpha^2\wedge\beta^2-BAu_{xy}\alpha^2\wedge \beta^2
-ACu_{xy}\beta^1\wedge \beta^2+ACu_{xy}\alpha^2\wedge \alpha^1\\
+CBu_{yy}\alpha^2\wedge \alpha^1+C^2u_{yy}\beta^1\wedge \alpha^1-(Au_x+Bu_y)\left(\lambda\,\alpha^1\wedge \beta^1+\mu\alpha^2\wedge\beta^1\right)-d(Bu_{y}\beta^2)\,.
\end{multline*}
Now, 
$$
(Au_x+Bu_y)\left(\lambda\,\alpha^1\wedge \beta^1+\mu\alpha^2\wedge\beta^1\right)=\lambda(Au_x+Bu_y)\,\alpha^1\wedge \beta^1+d(\mu u\,\beta^1)
$$
and we can write 
\begin{multline*}
-dJdu=(-C^2u_{yy}-\lambda Au_x-\lambda Bu_y)\alpha^1\wedge \beta^1+
(-A^2u_{xx}-BAu_{xy})\alpha^2\wedge \beta^2\\
-ACu_{xy}\beta^1\wedge \beta^2-(ACu_{xy}+B^2u_{yy})\alpha^1\wedge \alpha^2-d(\mu u\beta^1+Bu_{y}\beta^2)
\end{multline*}
which implies the first part of the statement. 

Moreover, 
\begin{multline*}
(\Omega+d\alpha)^2=\left((1-A^2u_{xx})(1-C^2u_{yy}-\lambda Au_x-\lambda Bu_y)-(ACu_{xy})^2\right)\, \Omega^2\\
=\frac{1}{l_1l_2}\left((l_1+u_{xx})(l_2+u_{yy}+m_1u_x+m_2u_y)-(u_{xy})^2\right)\, \Omega^2
\end{multline*}
where 
$$
l_{1}=-\frac{1}{A^2}\,,\quad l_{2}= -\frac{1}{C^2}\,,\quad m_1=-\lambda \frac{A}{C^2}\,,\quad m_2=-\lambda \frac{B}{C^2}
$$ 
and the claim follows. 
\end{proof}
From propositions \ref{PROP1} and \ref{PROP2} it follows that if we see $M$ as $2$-torus over a $2$-torus and we fix an invariant Lagrangian almost-K\"ahler structure $(\Omega,J)$ on $M$; then for every given $F$ defined on the base of $M$ and satisfying $\int_M {\rm e}^F\, \Omega^2=\int_M \Omega^2$ the corresponding Calabi-Yau equation can be written in terms of an unknown function $u$ on the base $\mathbb T^2_{xy}$ of $M$ as 
$$
\frac{1}{l_1l_2}\left((l_1+u_{xx})(l_2+u_{yy}+m_1u_{x}+m_2u_{y})-(u_{xy})^2\right)={\rm e}^F
$$
where $l_1,l_2,m_1,m_2\in \R$ and $l_1$ and $l_2$ are both positive or negative. This kind of equations are solvable in view of  theorem 6.2 in \cite{FLSV}. 

\section{The equation for non-invariant almost-complex structures}\label{S3}
As pointed out in \cite{TW} it is interesting to extend the results described in the previous section to  torus-invariant almost complex structures on the Kodaira-Thurston manifold $M$ which are compatible to \lq\lq natural\rq\rq\, symplectic form    
$\Omega_0$ defined in \eqref{standard}. In this section we consider some basic cases.   
Let $h=h(x_1,y_1)$ be a function in $C^{\infty}(\mathbb T^2_{x_1y_1})$ and consider the family of $\Omega_0$-compatible almost-complex structures $J_{h}$ induced by the relations 
\begin{equation}
J_{h}(e^1)=-{\rm e}^h\,f^1\,\quad J_{h}(e^2)=-f^2\,. 
\end{equation}
The following result is a generalization of proposition \ref{propBFV} to the family $J_{h}$. 

\begin{prop}\label{propBFVh}
Let $u\colon M\to\R$ be an $S^1$-invariant function and 
$$
\alpha:=-J_{h}du-ue^1. 
$$
Then 
$$
d\alpha \mbox{ is of type }(1,1)
$$ 
and 
\begin{equation}\label{CYspecific1}
(\Omega_0+d\alpha)^2=\left(\det(I+\mathcal{A}_{h}(u))-{\rm e}^{-h}u_{x_2y_1}^2\right)\,\Omega_0^2
\end{equation}
where $I$ is the identity $2\times 2$ matrix and 
\begin{equation}\label{A(u)hk}
\mathcal{A}_{h}(u)=\left(
\begin{array}{cc}
{\rm e}^h u_{x_1x_1}+{\rm e}^{-h}u_{y_1y_1}+u_{y_1}+ {\rm e}^h h_{x_1}u_{x_1}-{\rm e}^{-h} h_{y_1} u_{y_1} & u_{x_1x_2}\\
{\rm e}^{h}u_{x_1x_2} & u_{x_2x_2}
\end{array}
\right)
\end{equation}
\end{prop}
\begin{proof}
Let $u$ be an $S^1$-invariant function. Then 
$$
-J_{h}du={\rm e}^h u_{x_1}f^1+u_{x_2}f^2-{\rm e}^{-h} u_{y_1}e^1
$$
and 
$$
\begin{aligned}
-dJ_{h}du=&\,({\rm e}^hu_{x_1})_{x_1}e^1\wedge f^1+{\rm e}^hu_{x_1x_2}e^2\wedge f^1+u_{x_1x_2}e^1\wedge f^2+u_{x_2x_2}e^2\wedge f^2\\
&\,+u_{x_2y_1}f^1\wedge f^2+{\rm e}^{-h}u_{x_2y_1}e^1\wedge e^2+({\rm e}^{-h} u_{y_1})_{y_1}e^1\wedge f^1-u_{x_2}e^1\wedge e^2\,, 
\end{aligned}
$$
i.e.,
$$
\begin{aligned}
-dJ_{h}du=&\,\left({\rm e}^h u_{x_1x_1}+{\rm e}^{-h}u_{y_1y_1}+ {\rm e}^h h_{x_1}u_{x_1} -{\rm e}^{-h} h_{y_1} u_{y_1}+u_{y_1}\right) e^1\wedge f^1+u_{x_2x_2}e^2\wedge f^2\\
&\,+{\rm e}^hu_{x_1x_2}e^2\wedge f^1+u_{x_1x_2}e^1\wedge f^2+u_{x_2y_1}f^1\wedge f^2+\left({\rm e}^{-h}u_{x_2y_1}-u_{x_2}\right)e^1\wedge e^2\,.  
\end{aligned}
$$
Therefore if $\alpha=-J_{h}du-u e^1$, then 
$$
\begin{aligned}
d\alpha=&\,\left({\rm e}^h u_{x_1x_1}+{\rm e}^{-h}u_{y_1y_1}+ {\rm e}^h h_{x_1}u_{x_1} -{\rm e}^{-h} h_{y_1} u_{y_1}+u_{y_1}\right) e^1\wedge f^1+u_{x_2x_2}e^2\wedge f^2\\
&\,+{\rm e}^hu_{x_1x_2}e^2\wedge f^1+u_{x_1x_2}e^1\wedge f^2+u_{x_2y_1}f^1\wedge f^2+{\rm e}^{-h}u_{x_2y_1}e^1\wedge e^2\,.  
\end{aligned}$$
which is of type $(1,1)$ and  
$$
(\Omega_0+d\alpha)^2=\det(I+\mathcal{A}_{h}(u))-{\rm e}^{-h}u_{x_2y_1}^2\,,
$$
as required. 
\end{proof}
In view of proposition \ref{propBFVh}, the Calabi-Yau equation on $(M,\Omega_0,J_{h})$, for an $S^1$-invariant
function $F\colon M\to \R$ can be reduced to 
\begin{equation}\label{CYwarped}
\det(I+\mathcal{A}_{h}(u))-{\rm e}^{-h}\,u_{x_2y_1}^2={\rm e}^F
\end{equation}
where $\mathcal{A}_{h}$ is given by \eqref{A(u)hk} and $u\colon M\to \R$ is an unknown $S^1$-invariant function. Note that for $h=0$, equation \eqref{CYwarped} reduces to 
equation \eqref{det A} studied in \cite{JDG}.  We consider the following special cases: 

\medskip 
If $h=h(x_1)$ and $F=F(x_1,x_2)$ we may assume $u$ depending only on $(x_1,x_2)$ and \eqref{CYwarped} reduces in the variables $x=x_1$, $y=x_2$ to 
$$
\det\left(
\begin{array}{cc}
1+{\rm e}^h u_{xx}+ {\rm e}^h h'u_{x} & u_{xy}\\
{\rm e}^h\, u_{xy} & 1+u_{yy}
\end{array}
\right)={\rm e}^F
$$
on the $2$-dimensional torus $\mathbb T^2_{xy}$.  Such an equation can be rewritten as 
$$
\det\left(
\begin{array}{cc}
{\rm e}^{-h}+ u_{xx}+ h'u_{x} & u_{xy}\\
 u_{xy} &  1+u_{yy}
\end{array}
\right)={\rm e}^{F-h}\,.
$$

\medskip 
If $h=h(y_1)$ and $F=F(x_2,y_1)$,  then we assume $u$ depending only on $(x_2,y_1)$ and \eqref{CYwarped} reduces  in the variables $x=y_1$,  $y=x_2$ to 
$$
\det \left(
\begin{array}{cc}
1+{\rm e}^{-h}u_{xx}+(1-{\rm e}^{-h} h')u_{x} & u_{xy}\\
{\rm e}^{-h} u_{xy} &1+u_{yy}
\end{array}
\right)={\rm e}^F
$$
on $\mathbb T^2_{xy}$. Such an equation can be rewritten as 
$$
\det \left(
\begin{array}{cc}
{\rm e}^h+u_{xx}+({\rm e}^h-h')u_{x} & u_{xy}\\
u_{xy} & 1+u_{yy}
\end{array}
\right)={\rm e}^{F+h}\,. 
$$
 
\medskip 
Both cases fit in the following class of equations on $\mathbb T^2_{xy}$
$$
\det \left(
\begin{array}{cc}
{\rm e}^{-h}+u_{xx}+(c{\rm e}^{-h}+h')u_{x} & u_{xy}\\
u_{xy} & 1+u_{yy}
\end{array}
\right)={\rm e}^{F-h}
$$
where $h=h(x)$ is a smooth $1$-periodic functions on $\mathbb R$ and $c\in\mathbb R$.  We will show the solvability of the last class of equations in the next section.

\section{Solvability of the special cases}\label{S4}
The aim of this section is to prove the following result

\begin{theorem}\label{solution}
Let $h=h(x)$ be a smooth $1$-periodic functions on $\mathbb R$, $c\in \mathbb R$ and let $F=F(x,y)\in C^{\infty}(\mathbb T^2)$ be such that 
$$
\int_{\mathbb T^2}{\rm e}^F dx\wedge dy=1\,.
$$
Then  equation
\begin{equation}\label{warped}
\det\left(
\begin{array}{cc}
{\rm e}^{-h}+ u_{xx}+(c{\rm e}^{-h}+h')u_{x}  & u_{xy}\\
 u_{xy} & 1+u_{yy}
\end{array}
\right)={\rm e}^{F-h}
\end{equation}
has a solution $u\in C^{\infty}(\mathbb T^2)$. 
\end{theorem}

Before proving theorem \ref{solution} we consider the following preliminary lemma which is a slight generalization of lemma  6.3 in \cite{FLSV}. 

\begin{lemma}\label{preliminarlemma}
Let $h,v\in C^{1}(\R)$ be $1$-periodic functions satisfying 
$$
{\rm e}^{h}v'+(c+{\rm e}^{h}h')v>-1\,.
$$
Assume there exists $s_0\in [0,1]$ such that $v(s_0)=0$; then 
$$
\|v\|_{C^0}\leq C\,,
$$
where $C$ is a constant depending only on $c$ and $h$. 
\end{lemma}
\begin{proof}
Let $G$ be a primitive of $c{\rm e}^{-h}+h'$ in $\mathbb R$. Since 
$$
v'+ (c{\rm e}^{-h}+h')v> - {\rm e}^{-h}\,,
$$
in terms of $G$ we have 
$$
{\rm e}^G(v'+G'v)> - {\rm e}^{G-h}\,,
$$
i.e. 
$$
\frac{d}{ds}({\rm e}^Gv)> - {\rm e}^{G-h}\,.
$$
Since $v(s_0)=0$, we have  
$$
\int_{s_0}^s \frac{d}{ds}({\rm e}^Gv) ds>-\int_{s_0}^s  {\rm e}^{G-h}d\tau \,,\,\, \mbox{ for every  } s\geq 1\,,
$$
which implies 
$$
v(s)>-{\rm e}^{-G(s)}\int_{s_0}^s  {\rm e}^{G-h}d\tau \,,\,\, \mbox{ for every } s\in [1,2]\,. 
$$
On the other hand 
$$
\int_{s}^{s_0} \frac{d}{ds}({\rm e}^Gv) ds>-\int_{s}^{s_0}  {\rm e}^{G-h}d\tau \,,\,\, \mbox{ for every  } s\leq 0\,,
$$
which implies 
$$
v(s)<{\rm e}^{-G(s)}\int_{s}^{s_0}  {\rm e}^{G-h}d\tau \,,\,\, \mbox{ for every  } s\in [-1,0]\,. 
$$
The claim follows since $v$ is $1$-periodic. 
\end{proof}
Now we can prove theorem $\ref{solution}$

\begin{proof}[Proof of Theorem $\ref{solution}$]
Fix $0<\alpha<1$ and  let $C^{2,\alpha}_0(\mathbb T^2)$ be the space of $C^{2,\alpha}$-functions $u$ on $\mathbb T^2$ satisfying 
$$
\int_{\mathbb T^2}  u\,dx\wedge dy=0\,. 
$$
Then we  consider the operator $T\colon C^{2,\alpha}_0(\mathbb T^2)\times [0,1]\to C^{0,\alpha}_0(\mathbb T^2)$ defined by 
$$
T(u,t)=\det\left(
\begin{array}{cc}
{\rm e}^{-h}+ u_{xx}+(c\,{\rm e}^{-h}+h')u_{x}  & u_{xy}\\
 u_{xy} & 1+u_{yy}
\end{array}
\right)-{\rm e}^{-h}\left(t{\rm e}^F+ 1-t\right)
$$
in order that $u\in C^{2,\alpha}_0(\mathbb T^2)$ solves \eqref{warped} if and only if $T(u,1)=0$.
Then we define the set 
$$
S:=\{t\in [0,1]\,\,:\,\,\mbox{ there exists } u\in C^{2,\alpha}_0(\mathbb T^2) \mbox{ such that } T(u,t)=0\}\,.
$$ 
Note that $S$ is not empty since $u\equiv 0$ satisfies $T(u,0)=0$. 
We will show that $1\in S$ by proving that $S$ is open and closed in $[0,1]$. In this way we get that \eqref{warped} has a solution $u$ in $C^{2,\alpha}(\mathbb T^2)$ and theorem 3 in  \cite{Nirenberg} implies that $u$ is in fact $C^{\infty}$. Note that if $(u,t)\in C^2_0(\mathbb T^2)\times [0,1]$ is such that $T(u,t)=0$, then the matrix 
$$
\mathcal{A}_h:=\det\left(
\begin{array}{cc}
{\rm e}^{-h}+ u_{xx}+(c{\rm e}^{-h}+h')u_{x}  & u_{xy}\\
 u_{xy} & 1+u_{yy}
\end{array}
\right)
$$
is positive-defined. Indeed, since $\int_{\mathbb T^2}{\rm e}^Fdx\wedge dy=0$, then $\mathcal A_{h}(u)$ is non-singular and at a minimum point of $u$ all the eigenvalues of $\mathcal A_h$ are positive.

Now we prove that $S$ is closed. First of all we observe that if $u\in C^2_0(\mathbb T^2)$ satisfies $T(u,t)=0$ for some $t\in [0,1]$,  then 
\begin{eqnarray}
\label{prima}&{\rm e}^hu_{xx}+(c+{\rm e}^hh')u_{x} >-1\,, \\
\label{seconda}&  1+u_{yy}>-1.
\end{eqnarray}
Indeed, since 
$$
(1+{\rm e}^hu_{xx}+(c+{\rm e}^hh')u_{x})(1+u_{yy})>0
$$
the two terms have the same sign, and  they are both positive at a point $(x_0,y_0)$ where $u$ reaches its minimum value. Lemma \ref{preliminarlemma} then implies 
\begin{equation}\label{C1}
\|u_x\|_{C^0}\leq C \,\,\mbox{ and }\,\, \|u_y\|_{C^0}\leq C
\end{equation}
where $C$ is a constant depending on $c,h$ and $k$. 
Now we focus on the $C^0$ estimate on $u$. Let $(x_0,y_0)$ be a point in $[0,1]\times [0,1]$ where $u$ vanishes, then
\begin{multline*}
u(x,y)=(x-x_0)\,\int_0^1 u_x((1-t)x+tx_0,(1-t)y+ty_0)\,dt  +\\
(y-y_0)\int_0^1 u_y((1-t)x+tx_0,(1-t)y+ty_0))\,dt,
\end{multline*}
and by using \eqref{C1} we get 
$$
|u(x,y)|\leq C(x-x_0)  +C (y-y_0)
$$
which readily implies 
$$
\|u\|_{C^0}\leq C\,. 
$$
Hence $u$ satisfies a $C^1$ a priori bound. Furthermore, if $t\in [0,1]$ is fixed, equation 
$$
T(u,t)=0
$$
belongs to the class of equations studied in \cite{Heinz} and theorem 2 in \cite{Heinz}  implies that if $u\in C^{2,\alpha}_0(\mathbb T^2)$ solves $T(u,t)=0$ for some $t$ and satisfies a priori 
$C^1$ bound, then it also satisfies a $C^{2,\alpha}$ bound. 
This implies that $S$ is closed in $[0,1]$. Indeed, let $t_n$ be a sequence in $S$ converging to $\bar t$ in $[0,1]$. To each $t_n$ corresponds a function $u_n\in C^{2,\alpha}_0(\mathbb T^2)$ such that  $T(u_n,t_n)=0$.    
The $C^{2,\alpha}$ a priori bound on solutions to $T(u,t)=0$ implies that the sequence $u_n$ is bounded in $C^{2,\alpha}_0(\mathbb T^2)$ and so it admits a subsequence, which we still denote by $u_n$, which converges in $C^{2}_0(\mathbb T^2)$ to a function $\bar u\in C^{2}_0(\mathbb T^2)$.  Since $T$ is continuos, $T(\bar u,\bar t)=0$ and so, in view of \cite{Heinz}, $\bar u$ in $C^{2,\alpha}(\mathbb T^2)$.  Hence $\bar t\in S$ and $S$ is closed.  

Next we show that $S$ is open . Let $t_0\in S$. Then there exists $u\in C^{2,\alpha}_0(\mathbb T^2)$ such that $T(u,t_0)=0$. Let $L\colon C^{2,\alpha}_0(\mathbb T^2)\to C^{0,\alpha}_0(\mathbb T^2)$ be defined as 
$$
L(w):=T_{*|(u,t_0)}(w,0)\,. 
$$
A direct computation yields that 
\begin{multline}
L(w)= (w_{xx}+(c{\rm e}^{-h}+h')w_{x}) ( 1+u_{yy})\\+({\rm e}^{-h}+ u_{xx}+(c\,{\rm e}^{-h}+h')u_{x}) (w_{yy})-2u_{xy}w_{xy}
\end{multline}
and so $L$ is uniformly elliptic.  $L$  is injective by maximum principle and it is surjective in view of elliptic theory (see e.g. \cite{Gilbarg-Trudinger}). Therefore the implicit function theorem implies that  $\bar t$ has a open neighborhood contained in $S$, and so $S$ is open, as required.      
\end{proof}

\section{A generalization to higher dimensions}
In this section we consider a generalization of the Kodaira-Thurston manifold in dimension greater than $4$. 
Assume $n\geq 3$. Let $G_n$ be the Lie group $(\R^{2n},*_{n})$, where 
\begin{multline*}
(x_1,\dots,x_{n},y_1,\dots,y_{n})*_n(x'_1,\dots,x'_{n},y'_1,\dots,y'_{n})=\\
(x_1+x'_1,\dots,x_{n}+x'_{n},y_1+y'_1,y_2+y'_2-x_2x_1',\dots,y_{n-1}+y_{n-1}'-x_{n}x_{1}')
\end{multline*}
and let $M_n=\Gamma_n\backslash G_n$, where $\Gamma_n$ is $\mathbb Z^{2n}$ with the multiplication induced by $*_n$.  
Then $M_n$ is a $2$-step nilmanifold and the projection $\pi\colon \R^{2n}\to \R^{n+1}$ onto the first $(n+1)$-coordinates induces to $M_n$ a structure of principal $(n-1)$--torus bundle on an $(n+1)$--torus $\mathbb T^{n+1}$.  
$M_n$ is parallelizable and 
$$
e^i=dx_i\,,\quad i=1\dots,n,\quad f^j=dy_j-x_1dx_j\,,\quad j=1\dots,n
$$
defines a global coframe which satisfies 
$$
de^k=0\,,\,\,k=1,\dots,n\,,\quad df^1=0\,,\quad  df^k=e^{k}\wedge e^1\,,\quad k=2,\dots, n\,.
$$
We then consider on $M_n$ the symplectic form 
$$
\Omega_n=\sum_{k=1}^n\alpha^k\wedge \beta^k
$$
and the $\Omega_n$-compatible almost-complex structure $J_n$ induced by $\Omega_n$ and the natural metric 
$$
g_n=\sum_{k=1}^n\alpha^k\otimes \alpha^k+\beta^k\otimes \beta^k\,.
$$ 
In terms of the basis $\mathcal B=\{e^1,\dots,e^n,f^1,\dots,f^n\}$, $J_n$ is defined by 
$$
J_n e^k=-f^k\,,\quad J_n f^k=e^k\,.
$$
Let $u$ be a $T^{n+1}$-invariant function on $M_n$; then 
$$
du=\sum_{s=1}^n u_{x_s}\,e^s+u_{y_1} f^1\,,\quad -J_ndu=\sum_{s=1}^n u_{x_s}\,f^s-u_{y_1} e^1
$$
and so 
$$
\begin{aligned}
-dJ_ndu=&\sum_{r,s=1}^{n}u_{x_rx_s}e^{r}\wedge f^{s}-\sum_{k=1}^nu_{x_ky_1}e^k\wedge e^1+u_{x_ky_1}f^1\wedge f^k+u_{y_1y_1}e^1\wedge f^1+\sum_{k=2}^{n}u_{x_r}e^r\wedge e^1\\
=&\sum_{r,s=1}^{n}u_{x_rx_s}e^{r}\wedge f^{s}-\sum_{k=1}^nu_{x_ky_1}e^k\wedge e^1+u_{x_ky_1}f^1\wedge f^k+(u_{y_1y_1}-u_{y_1})e^1\wedge f^1+d(ue^1)\,,
\end{aligned}
$$
and so if  
$$
\alpha=-J_ndu-ue^1\,,
$$
then $d\alpha$ is of type $(1,1)$ with respect to $J_n$. Furthermore, 
\begin{multline*}
(\Omega_n+d\alpha)^n=\left(\sum_{r,s=1}^{n}(\delta_{rs}+u_{x_rx_s})e^{r}\wedge f^{s}+(1+u_{y_1y_1}-u_{y_1})e^1\wedge f^1\right)^n\\
-n!\sum_{k,m=2}^n\left(\prod _{r,s=2,(r,s)\neq (k,m) }^n\, u_{x_rx_s}u_{x_ky_1}u_{x_my_1}\right)e^1\,\wedge f^1\wedge\dots\wedge e^n\wedge f^n
\end{multline*}
and if $F$ is a given $T^{n+1}$-invariant map, the Calabi-Yau equation $(\Omega+d\alpha)^n={\rm e}^F\,\Omega^n$ reads in terms of $u$ as 
\begin{equation}\label{ndimCY}
{\rm det}(I+\mathcal A(u))-\sum_{k,m=2}^n\left(\prod _{r,s=2, (r,s)\neq (k,m) }^n\, u_{x_rx_s}u_{x_ky_1}u_{x_my_1}\right)={\rm e}^F\,,
\end{equation}
where $\mathcal A(u)=(A_{ij})$ is the $n\times n $ matrix  
$$
A_{11}=u_{x_1x_1}+u_{y_1y_1}-u_{y_1}\,,\quad A_{ij}=u_{x_ix_j}\,,\quad \mbox{ if }( i,j)\neq (1,1)\,.  
$$
\begin{ex}
{\em For $n=3$, equation \eqref{ndimCY} reads as 
$$
{\rm det}(I+A(u))-u_{x_3x_3}u_{x_2y_1}^2-u_{x_2x_2}u_{x_3y_1}^2-2u_{x_2x_3}u_{x_2y_1}u_{x_3y_1}={\rm e}^F\,,
$$
this kind of equations has been considered in \cite{Giaretti}. 
}
\end{ex}

In analogy to the case $n=2$, we can obtain special cases by regarding $M_n$ as a principal $T^{n}$-bundle over a $\mathbb T^n$ and assuming  $F$ to be $T^n$-invariant.  It is not restrictive considering only the following two cases:
$$
F=F(x_1,\dots,x_n)\,,\quad \mbox{ or } F=F(x_2,\dots,x_n,y_1).
$$

\begin{itemize}
\item In the first case $F=F(x_1,\dots,x_n)$, equation \eqref{ndimCY} reduces to the Monge-Amp\`ere equation 
$$
\det(I+\mathcal H(u))={\rm e}^F
$$
on the $n$-dimensional torus $\mathbb T^n= \R^n/\mathbb Z^n$, where $\mathcal H(u)$ is the Hessian metric of $u$.   In this case the equation has a solution in view of \cite{YY}. 

\medskip 
\item In the second case, $F=F(x_2,\dots,x_n,y_1)$, in the variables 
$$
z_1=y_1\,, z_2=x_2,\dots, z_{n}=x_n
$$
equation \eqref{ndimCY} take the following expression 
$$
\det(I+\mathcal B(u))={\rm e}^F
$$ 
where $\mathcal B(u)=(B_{ij})$ is given by  
$$
B_{11}=u_{z_1z_1}+u_{z_1}\,,\quad B_{ij}= u_{z_iz_j}\,, \mbox{ if } i,j\neq 1\,. 
$$
\end{itemize}

%
%
%
%
%
%
%


\begin{thebibliography}{1}

\bibitem{REN}
E. Buzano, A. Fino, L. Vezzoni, The Calabi-Yau equation for $T^2$-bundles over the non-Lagrangian case. {\em Rend. Semin. Mat. Univ. Politec. Torino} {\bf 69} (2011), no. 3, 281--298.

\bibitem{JDG}
E. Buzano, A. Fino, L. Vezzoni, The Calabi-Yau equation on the Kodaira-Thurston manifold, viewed as an $S^1$-bundle over a $3$-torus. {\em J. Differential Geom.} {\bf 101} (2015), no. 2, 175--195.

\bibitem{D} S. K. Donaldson, Two-forms on four-manifolds
  and elliptic equations. {\em Inspired by S.S.~Chern},
153--172, Nankai Tracts Math. 11, World Scientific, Hackensack NJ, 2006.

\bibitem{FLSV} {A. Fino, Y.Y.  Li, S. Salamon, L. Vezzoni,}
The Calabi--Yau equation on $4$-manifolds over $2$-tori,  {\em Trans. Amer. Math. Soc.} {\bf 365} (2013), no. 3, 1551--1575.

%
%


\bibitem{Gilbarg-Trudinger}
{D. Gilbarg,  N.S. Trudinger,} \emph{{Elliptic partial differential
  equations of second order}}, second ed., Grundlehren der Mathematischen
  Wissenschaften [Fundamental Principles of Mathematical Sciences], vol. 224,
  Springer-Verlag, Berlin, 1983.

\bibitem{Giaretti}
M. Giaretti: {\em Equazione di Calabi-Yau in geometria simplettica (Calavi-Yau equation in Symplectic Geometry)}, Master Thesis, Universit\`a degli studi di Torino,  2012.

\bibitem{Heinz}
{E. Heinz}, Interior estimates for solutions of elliptic
  {M}onge-{A}mp\`ere equations, {\em Proc. {S}ympos. {P}ure {M}ath., {V}ol.
  {IV}}, American Mathematical Society, Providence, R.I., 1961, 149--155.

\bibitem{YY}
Y. Y. Li, Some existence results of fully nonlinear elliptic equations of Monge-Amp\`ere type, {\em Comm.
Pure Appl. Math.} {\bf 43} (1990), 233--271.


\bibitem{Nirenberg}
{L. Nirenberg},  On nonlinear elliptic partial differential equations
  and {H}\"older continuity, {\em Comm. Pure Appl. Math.} \textbf{6} (1953),
  103--156; addendum, 395.

\bibitem{TWWY}
V. Tosatti, Y. Wang, Yu, B. Weinkove, X. Yang, 
$C^{2,\alpha}$ estimates for nonlinear elliptic equations in complex and almost complex geometry. 
{\em Calc. Var. Partial Differential Equations} {\bf 54} (2015), no. 1, 431--453. 


\bibitem{TWY} V. Tosatti, B. Weinkove, S.T. Yau, Taming symplectic
  forms and the Calabi-Yau equation, {\em Proc. London Math. Soc.}
  {\bf 97} (2008), no. 2, 401--424.


\bibitem{TW} V. Tosatti, B. Weinkove, The Calabi-Yau equation on the
  Kodaira-Thurston manifold,    {\em J. Inst. Math.
    Jussieu} {\bf 10} (2011), no. 2, 437--447.


\bibitem{taubes}
C. H. Taubes: Tamed to compatible: symplectic forms via moduli space integration, {\em J. Symplectic Geom.} {\bf 9} (2011), no. 2, 161--250.


\bibitem{Yau} S. T. Yau, On the Ricci curvature of a compact K\" ahler
  manifold and the complex Monge-Amp\`ere equation, I, \emph{Comm. Pure
    Appl. Math.} {\bf 31} (1978), no. 3, 339--411.

\bibitem{W}
B. Weinkove, The Calabi-Yau equation on almost-K\"ahler four-manifolds, {\em J. Differential
Geom.} {\bf 76} (2007), no. 2, 317--349. 

\end{thebibliography}
\end{document}